\documentclass[reqno]{amsart}

\usepackage{amssymb}
\usepackage{hyperref}
\usepackage{tikz}

\newtheorem{theorem}{Theorem}[section]
\newtheorem{lemma}[theorem]{Lemma}

\newtheorem{corollary}[theorem]{Corollary}

\theoremstyle{definition}
\newtheorem{definition}[theorem]{Definition}

\theoremstyle{remark}
\newtheorem{remark}[theorem]{Remark}
\newtheorem{example}[theorem]{Example}

\newcommand{\la}{\lambda}

\DeclareMathOperator*{\esssup}{ess\,sup}

\title[Optimal partitioning of an interval and applications\dots]{Optimal partitioning of an interval and applications to Sturm-Liouville eigenvalues}
\author{Paolo Tilli and Davide Zucco}

\address{Paolo Tilli, Dipartimento di Scienze Matematiche, Politecnico di Torino, Italy}
\email{paolo.tilli@polito.it}

\address{Davide Zucco, Istituto Nazionale di Alta Matematica, Unit\`a di Ricerca del Dipartimento di Scienze Matematiche, Politecnico di Torino, Italy}
\email{davide.zucco@polito.it}

\begin{document}

\begin{abstract}
We study the optimal partitioning of a (possibly unbounded) interval of
the real line into $n$ subintervals in order to minimize the maximum
of certain set-functions, under rather general assumptions such as
continuity, monotonicity, and a Radon-Nikodym property.
We prove existence and uniqueness of a solution to this minimax partition problem, showing that the values of the 
set-functions on the intervals of any optimal partition must coincide. We also investigate 
the asymptotic distribution of the optimal partitions as $n$ tends to infinity.
Several examples of set-functions fit in this framework, including measures,
weighted distances and eigenvalues. We recover, in particular, some classical 
results of Sturm-Liouville theory: the asymptotic distribution of the zeros of the 
eigenfunctions, the asymptotics of the eigenvalues, and the celebrated Weyl law on the 
asymptotics of the counting function.

\smallskip

\end{abstract}

\maketitle

\section{Introduction}

The problem of best partitioning a measurable space has received lot of attention in the last decades, since its importance from theory to applications. On one side, rigorous studies of partition problems require advanced mathematical tools; on the other, optimal partitions arise in several concrete applications, such as in discrete allocation problems, in statistical decision theory, and in phenomena of spatial segregation in reaction-diffusion systems.
For some works on the subject one can consult \cite{bonhel,bfvv,caflin,coteve,dalewi,dubspa,elhike, hehote, lewi} and the references therein. 

The present paper starts from the easy consideration that most of the results in the literature are set up for optimal partition problems in a general higher-dimensional framework. Our aim is, on the contrary, to focus on optimal partition problems in \emph{one dimension} (in fact this is  an ongoing project that we initiated in \cite{tilzuc} looking for spectral partitions that minimize the sum of the eigenvalues of certain Sturm-Liouville problems). A crucial fact in one dimension is that each partition of an interval may be identified by the points that induce the partition. In these terms, an optimal partition problem can be equivalently regarded as an \emph{optimal location problem}, namely with  points as unknowns (see, for instance, \cite{bojima,busava, lmsz,suzdre,suzoka}).

To set up our framework let us introduce some notation. Given $-\infty\leq a<b\leq +\infty$ and a natural number $n$ we denote by $I:=(a,b)$ a generic (possibly unbounded) open subinterval of $\mathbb R$ (one may also consider the whole real line or a half-line) and by
\[
\mathcal C_n(I):=\big\{\{I_j\}:\, I_j=(x_{j-1},x_j)\text{ with $x_{j-1}\leq x_j$ for $1\leq j\leq n$, $x_0:=a$, $x_{n}:=b$}\big\},
\]
the class of partitions of $I$ made up of $n$ open intervals (notice that some intervals of a partition in $\mathcal C_n(I)$ may be empty, while the non-empty ones are disjoint).
Moreover, we consider a family of set-functions $\{f_j\}_{j=1}^n$, each function $f_j$ being defined on the subintervals of $I$ and satisfying some abstract conditions (among them we require a monotonicity with respect to domain inclusion and, for the asymptotics, a Radon-Nikodym condition).
Then, to each interval $I_j$ of a partition in $\mathcal C_n(I)$, we may associate the real number $f_j(I_j)$, and study the problem of minimizing the maximum of the $f_j(I_j)$: in formulae 
\begin{equation}\label{problem}
\min_{\{I_j\}\in \mathcal C_n(I)}\max_{1\leq j\leq n} f_j(I_j).
\end{equation}

Due to the abstract framework it is clearly not possible to characterize every solution of this \emph{minimax partition problem} just in terms of some geometrical quantities. 
As a first step it is then important to investigate questions concerning the existence of an optimal partition, its characterization via some optimality condition, the uniqueness, as well as the asymptotic distribution of the minimizing sequences as the number of intervals of the partition goes to infinity. By taking advantage of the one dimensional framework we are able to give a satisfactory answer to these issues.

In Section~\ref{sec.2} we focus on the problem when the number $n$ of intervals of the partition is fixed. As a preliminary step of our analysis we introduce the hypotheses on the functions $f_j$ that guarantee the existence of an optimal partition: continuity, domain monotonicity and a compatibility condition for consecutive functions.
Then we derive the optimality condition \eqref{optimality}: minimizers of \eqref{problem} are characterized by the fact that the values $f_j(I_j)$ of the functions on the intervals of the partition must coincide.
This optimality condition is very robust since it allows to obtain useful information on the minimizers. As a first companion result of  \eqref{optimality} we deduce, in the case of strictly monotone set-functions, the uniqueness of an optimal partition. All these results are summarized in Theorem~\ref{teo.main}. Moreover, for specific choices of the functions $f_j$, we can find out as solution the uniform partition, see Corollary~\ref{cor.uniform}. As a byproduct of the general monotonicity assumptions we prove in Corollary \ref{cor.maxmin} the equivalence of \eqref{problem} with the \emph{maximin partition problem}
\begin{equation}\label{problem2}
\max_{\{I_j\}\in \mathcal C_n(I)}\min_{1\leq j\leq n} f_j(I_j).
\end{equation}

In Section~\ref{sec.3} we then analyze the asymptotics of the minimizers. As $n$ increases, the points that identify the optimal partition tend to fill the whole interval $I$, and in the limit any information concerning the density (i.e. number of these points for unit length) is lost.  The common strategy used to retrieve this information is to prove a $\Gamma$-convergence result in the space of probability measures. However, thanks again to the optimality condition \eqref{optimality} we are able to bypass the technicalities of $\Gamma$-convergence (cf. with \cite{tilzuc} where this strategy was not possible).  We prove in Theorem~\ref{teo.main2} the asymptotics of the minimizers and identify the distribution of the points in the limit. 

Both sections are supplied by several set-functions that satisfy the hypotheses that we require for the validity of the results: measures, weighted distances, and eigenvalues, each one originating and connecting different optimization problems.
The last Section~\ref{sec.5} contains a remarkable example. We study a maximization problem for the first eigenvalue of a Sturm-Liouville problem (see \cite{henzuc,tilzuc,tilzuc2} for related two dimensional versions of this problem) and we derive by purely variational techniques, some classical results of the Sturm-Liouville theory: the asymptotic distribution of the zeros of the eigenfunctions, the asymptotics of the eigenvalues and the well-celebrated \emph{Weyl law}, concerning the asymptotics of the counting function.

Eventually, we point out that most of our proofs are constructive, following \emph{iterative criterions} which could be numerically implemented to find out the solution in more complex situations.

\smallskip

\paragraph{\textbf{Notation}} The word interval always stands for \emph{open} interval. The symbols $\subseteq$ and $\subsetneq$ are the inclusion and the strict inclusion among sets.  The limit $\lim_{J\downarrow x}$ is computed along sequences of \emph{bounded} intervals $J$ that shrink around the point $x$. The sequence $\{a_j^n\}$ stands for $\{a_j^n\}_j$, i.e. a sequence with respect to the parameter $j$ while $n\in\mathbb N$ is fixed. For $r\in [1,+\infty]$  the classical \emph{Lebesgue} and \emph{Sobolev spaces} are respectively denote by $L^r(I)$ and $W^{1,r}(I)$. The \emph{counting measure} and the \emph{Lebesgue measure} are respectively denoted by $\sharp$ and $\mathcal L$.

\section{Optimal partitions for fixed $n$}\label{sec.2}

We study in this section the minimax problem \eqref{problem} when the number $n$ of intervals of the partition is \emph{fixed}. The family of set-functions $\{f_j\}$ considered in the problem must satisfy suitable abstract conditions, that we introduce in the following with some notation.

\begin{definition}[Set-function]\label{def.set}
A \emph{set-function} $f$ is a function defined on the open subintervals of $I$ whose range is in $[0,+\infty]$.
\end{definition}

\begin{remark}\label{rem.pos}
A set-function is tacitly assumed to be positive with $+\infty$ as admissible value. With minor changes, all the results of the paper should hold without this positivity assumption (we assume it to simplify the exposition).
\end{remark}

\begin{definition}[Continuity]\label{def.con}
A set-function $f$ is said \emph{continuous} if the function of two variables that to every $(x,y)\in \overline I\times \overline I$ with $x\leq y$ associates the quantity $f((x,y))$ is continuous.
\end{definition}

In other words, a set-function $f$ is continuous with respect to the movements of the endpoints of the interval: if $(x_n,y_n)\to (x,y)$ then $f((x_n,y_n))\to f((x,y))$.

\begin{definition}[Monotonicity]\label{def.mon}
A set-function $f$ is said \emph{increasing} if $f(I_1)\leq f(I_2)$ for all intervals $I_1\subseteq I_2\subseteq I$.
A set-function $f$ is said \emph{decreasing} if, on the contrary, $f(I_1)\geq f(I_2)$ for all intervals $I_1\subseteq I_2\subseteq I$. A  set-function is simply said \emph{monotone} if it is increasing or decreasing. When the previous inequalities are always strict (except of course for the case $I_1=I_2$) we specify and write strictly increasing, strictly decreasing, or  strictly monotone set-functions.
\end{definition}

\begin{remark}[Measures]\label{rem.meas}
The notion of continuity in Definition~\ref{def.con} slightly differs from the one which usually is used in the context of Measure Theory.\footnote{A measure $\mu$ is said continuous (see for instance \cite[Theorem 1.2]{evagar} and \cite[Definition 2.37]{waki}) if $I_1\subseteq\dots \subseteq I_k\subseteq I_{k+1} \dots$ then $\lim_k(I_k)=\mu(\cup_k I_k)$ and if $I_1\supseteq\dots \supseteq I_k\supseteq I_{k+1} \dots$ with $\mu(I_1)< +\infty$ then $\lim_k\mu(I_k)=\mu(\cap_k I_k)$.} 
Indeed any Dirac delta supported at a point $x\in I$, and more generally any positive measure on $I$ with atoms, is not continuous according to Definition~\ref{def.con}. Also the Lebesgue measure $\mathcal L$ is not continuous on the entire real line, since e.g. $\lim_{x\to+\infty}\mathcal L((x,+\infty))\neq \mathcal L(\emptyset)$.
By the way, every non-atomic probability measure on $\overline I$ is a continuous and increasing set-function (it may not necessarily be strictly increasing). 
\end{remark}

Others important examples of continuous and monotone set-functions are the following ones.

\begin{remark}[Distances]\label{rem.dist}
Let $r\in [1,+\infty)$ and $\rho\in L^1(I)$ be continuous and positive (i.e. $\rho>0$) with $I$ bounded. Then one may consider the \emph{average distance} functional that to every subinterval $J\subseteq I$ associates the quantity
$$\int_J \rho(x) \mathrm{d}_{\partial J}(x)^r dx,$$
or consider the \emph{maximum distance} functional, its non-local counterpart which associates
$$\max_{x\in J }(\rho(x)\mathrm{d}_{\partial J}(x)),$$
where $\mathrm{d}_{\partial J}$ denote the distance function from the boundary $\partial J$. 
As set-functions these functionals are continuous and increasing, according to Definitions~\ref{def.con} and \ref{def.mon}.
\end{remark}

Measures have a stronger property: they are \emph{additive} set-functions.
Remarkable examples of non-additive set-functions are the following ones. 
 
\begin{remark}[Eigenvalues]\label{rem.eigen}
Let ${1}/{p},q,w\in L^1(I)$ with $p,w>0$ a.e. on $I$. Consider on a subinterval $J\subseteq I$ the self-adjoint \emph{Sturm-Liouville problem} consisting of the \emph{Sturm-Liouville equation}
\[
-(p u')'+qu=\la w u\quad  \text{on $J$},
\] 
with the self-adjoint boundary condition $u=0$ on $\partial J$ (these are the so-called \emph{Dirichlet boundary conditions}) with $\lambda\in \mathbb C$ the spectral parameter. It is well known \cite{atk} that the problem has infinitely many eigenvalues $\lambda_n$ (with associated eigenfunctions $u_n$), all of which are real, numbered to form a non-decreasing sequence,
\[
-\infty<\la_1\leq \la_2\leq  \dots\leq \la_n\leq\dots
\]
approaching $+\infty$. Then each eigenvalue $\la_n=\la_n(J)$ is a set-function depending on the interval $J$ according to Definition~\ref{def.set} and Remark~\ref{rem.pos} (indeed it becomes a positive set-function up to a suitable translation of $\la_1(I)$, independent of $J$). 
It is known \cite[Lemma 2.1]{kowuze} that each of these eigenvalues is continuous according to Definition~\ref{def.con} and
\cite[Remark 4.4.5]{zettl} that it is strictly decreasing according to Definition~\ref{def.mon}. More general separated boundary conditions can be considered (including the \emph{Neumann boundary conditions}) but for the monotonicity of the eigenvalues some further assumptions on the coefficients $p,q,w$ are needed (it suffices for instance the boundedness of $q/w$, see again \cite[Remark 4.4.5]{zettl}).
\end{remark}

Since in \eqref{problem} we allow different set-functions on each interval of a partition, a compatibility condition is also needed.  

\begin{definition}[Compatibility]\label{def.com}
Two monotone set-functions $f_{j_1}$ and $f_{j_2}$ are \emph{compatible} if they have the same monotonicity (they are both increasing or both decreasing) and if they assume the same value on the empty set, namely $f_{j_1}(\emptyset)=f_{j_2}(\emptyset)$.
\end{definition}

\begin{remark}\label{rem.com}
The compatibility condition is an equivalence relation, in particular a monotone set-function is compatible to itself.
\end{remark}

\begin{remark}\label{rem.equ}
Let $f_{j_1}$ and $f_{j_2}$ be compatible according to Definition~\ref{def.com}. 
In the case $f_{j_1}$ and $f_{j_2}$ are increasing then $f_{j_1}(\emptyset)\leq f_{j_2}(J)$ for all intervals $J\subseteq I$. On the contrary, in the case $f_{j_1}$ and $f_{j_2}$ are decreasing then 
$f_{j_1}(J)\leq f_{j_2}(\emptyset)$ for all intervals $J\subseteq I$.
\end{remark}

For fixed $n$ we have good information on the minimizer of \eqref{problem}.

\begin{theorem}[Existence, uniqueness, and the optimality condition]\label{teo.main}
Let $n\in\mathbb N$ with $n>1$ and $\{f_j\}_{j=1}^n$ be a family of set-functions such that, for every $1\leq j\leq n$,
\begin{itemize}
\item $f_j$ is continuous and monotone;
\item $f_j$ and $f_{j+1}$ are compatible.
\end{itemize}
There exists a solution to the minimax partition problem \eqref{problem}. A partition $\{I_j\}\in \mathcal C_n(I)$ solves \eqref{problem} if and only if
\begin{equation}\label{optimality}
f_{j_1}(I_{j_1})=f_{j_2}(I_{j_2}),\quad  \text{for all $1\leq j_1,j_2\leq n$.}
\end{equation}
If, in addition, all the set-functions $f_j$ are \emph{strictly} monotone, then the solution is unique.
\end{theorem}

The transitivity of the compatibility condition (see Remark~\ref{rem.com}) ensures the functions of the family $\{f_j\}$ to have the same monotonicity (they are all increasing or all decreasing).

To prove this theorem we need the following combinatorial result.

\begin{lemma}\label{lemma:incl}
Given $n\in \mathbb N$ with $n>1$, let $\{I_j\}$ and $\{\widehat I_j\}$ be partitions in $\mathcal C_n(I)$.
There exist two indices $j_1$ and $j_2$ such that
\[
I_{j_1}\subseteq \widehat I_{j_1} \qquad \text{and}\qquad \widehat I_{j_2}\subseteq I_{j_2}.
\]
Moreover, if $\{I_j\}$ and $\{\widehat I_j\}$ are distinct partitions, then one can require these inclusions to be strict.
\end{lemma}

\begin{proof}
Let $\{x_j\}$ and $\{\widehat x_j\}$ be the sets made up of $(n-1)$ points of $I$ that define the partitions $\{I_j\}$ and $\{\widehat I_j\}$, respectively.
Consider the sets  
\[
L:=\{j\in \mathbb N:\; 1\leq j\leq (n-1),\; x_j<\widehat x_j\}
\] 
and
\[
R:=\{j\in \mathbb N:\; 1\leq j\leq (n-1),\; x_j>\widehat x_j\}.
\]
Apply then the following procedure (see Figure~\ref{fig.com} for an example), recalling that by definition $x_0=\widehat x_0$ and $x_n=~\widehat x_n$.
\begin{itemize}
\item[-] If $L=R=\emptyset$ then $I_j=\widehat I_j$ for all $j$ and one may choose $j_1=j_2=1$.
\item[-] Else if $L\neq \emptyset$ then we may choose $j_1=\min L$ and $j_2=(\max L+1)$. Indeed in this case $x_{j_1}<\widehat x_{j_1}$ and $x_{j_2}\geq \widehat x_{j_2}$ which, together with the optimality of $j_1$ and of $j_2$, yield $x_{j_1-1}\geq \widehat x_{j_1-1}$ and $x_{j_2-1}< \widehat x_{j_2-1}$. 
\item[-] Else if $R\neq \emptyset$ then we may choose $j_1=(\max R+1)$ and $j_2=\min R$. Indeed in this case $x_{j_1}\leq\widehat x_{j_1}$ and $x_{j_2}> \widehat x_{j_2}$ which, together with the optimality of $j_1$ and of $j_2$, imply $x_{j_1-1}> \widehat x_{j_1-1}$ and $x_{j_2-1}\leq \widehat x_{j_2-1}$. 
\end{itemize}
Then $I_{j_1}=( x_{j_1-1},x_{j_1})\subseteq (\widehat x_{j_1-1}, \widehat x_{j_1})= \widehat I_{j_1}$ and $\widehat I_{j_2}=(\widehat x_{j_2-1},\widehat x_{j_2})\subseteq ( x_{j_2-1},  x_{j_2})=  I_{j_2}$.
Note that the inclusions can not be strict only in the case $L=R=\emptyset$.
\end{proof}

\begin{figure}[t]
\begin{tikzpicture}
\draw (0,0) -- (10,0);
\draw (0,-1) -- (10,-1);

\coordinate (a0) at (0,0) node at (0,-0.3) {\small $x_0$};
\coordinate (a1) at (1,0) node at (1,-0.3) {\small $x_1$};
\coordinate (a2) at (2,0) node at (2,-0.3) {\small $x_2$};
\coordinate (a3) at (3,0) node at (3,-0.3) {\small $x_3$};
\coordinate (a4) at (8.5,0) node at (8,-0.3) {\small $x_4$};
\coordinate (a5) at (10,0) node at (10,-0.3) {\small $x_5$};

\coordinate (b0) at (0,-1) node at (0,-1.3) {\small $\widehat x_0$};
\coordinate (b1) at (1,-1) node at (1,-1.3) {\small $\widehat x_1$};
\coordinate (b2) at (4,-1) node at (4,-1.3) {\small $ \widehat x_2$};
\coordinate (b3) at (6,-1) node at (6,-1.3) {\small $\widehat x_3$};
\coordinate (b4) at (8,-1) node at (8,-1.3) {\small $\widehat x_4$};
\coordinate (b5) at (10,-1) node at (10,-1.3) {\small $ \widehat x_5$};

\node at (0.5,0.2) {\small $I_1$};
\node at (1.5,0.2) {\small $I_2$};
\node at (2.5,0.2) {\small $I_3$};
\node at (5.5,0.2) {\small $I_4$};
\node at (9.25,0.2) {\small $I_5$};

\node at (0.5,-0.75) {\small $\widehat I_1$};
\node at (2.5,-0.75) {\small $\widehat I_2$};
\node at (5,-0.75) {\small $\widehat I_3$};
\node at (7,-0.75) {\small $\widehat I_4$};
\node at (9,-0.75) {\small $\widehat I_5$};

\draw [fill=black] (a0) circle (1.5pt);
\draw [fill=black] (a1) circle (1.5pt);
\draw [fill=black] (a2) circle (1.5pt);
\draw [fill=black] (a3) circle (1.5pt);
\draw [fill=black] (a4) circle (1.5pt);
\draw [fill=black] (a5) circle (1.5pt);

\draw [fill=black] (b0) circle (1.5pt);
\draw [fill=black] (b1) circle (1.5pt);
\draw [fill=black] (b2) circle (1.5pt);
\draw [fill=black] (b3) circle (1.5pt);
\draw [fill=black] (b4) circle (1.5pt);
\draw [fill=black] (b5) circle (1.5pt);

\end{tikzpicture}
\caption{Example of Lemma~\ref{lemma:incl} with $n=5$; here $j_1=2$ and $j_2=4$.}\label{fig.com}
\end{figure}
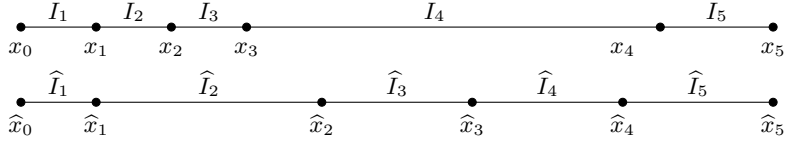

It is interesting to notice that the inclusions in Lemma~\ref{lemma:incl} hold for couples of intervals $I_j$ and $\widehat I_j$ labeled with the same indeces. This lemma is crucial in the proof of Theorem~\ref{teo.main}.

\begin{proof}[Proof of Theorem~\ref{teo.main}]
We first prove the existence of a solution and then that the optimality condition \eqref{optimality} is necessary and sufficient. Combining  \eqref{optimality} with the strict monotonicity assumption we finally derive the uniqueness of the solution.
\smallskip

\emph{Existence.} Let $\mathcal S\subseteq \mathbb R^{n-1}$ be the compact simplex defined by
\[
\mathcal S:=\left\{(x_1,\dots, x_{n-1})\in \mathbb R^{n-1}: \text{with $x_{j-1}\leq x_j$ for $1\leq j\leq n$, $x_0:=a$, $x_{n}:=b$}\right\}\!,
\] 
and let $F\colon \mathcal S\to [0,+\infty]$ be the function of $(n-1)$-real variables defined by
\[
F(x_1,\dots, x_n):=\max_{1\leq j\leq n}f_j((x_{j-1},x_{j})).
\] 
From the continuity of every set-function $f_j$ the function $F$ turns out to be continuous on $\mathcal S$. 
The existence of a solution of \eqref{problem} is then equivalent to the existence of a minimizer of the continuous function $F$ on the compact set $\mathcal S$. This last assertion is 
an immediate consequence of the Weierstrass extreme value theorem.

\smallskip
\emph{Necessary condition.} Let $\{I_j\}\in \mathcal C_n(I)$ be a solution of problem \eqref{problem}. Define the quantity $M:=\max_{1\leq j\leq n} f_j(I_j)$ and denote by $O$ the set of those $j$ corresponding to optimal intervals $I_j$ such that $f_j(I_j)=M$ and by $O^c$ its complement $\{1,\dots,n\}\setminus O$. By contradiction, assume that \eqref{optimality} does not hold, namely that  
$O^c\neq \emptyset$,
and apply the following procedure to reach the contradiction.
\begin{itemize}
\item[-] Consider an optimal interval $I_{j_1}$ with $j_1\in O$, which is adjacent to a non-optimal one $I_{j_2}$ with $j_2\in O^c$, thus by construction $f_{j_1}(I_{j_1})>f_{j_2}(I_{j_2})$. By Remark~\ref{rem.equ} if the functions $\{f_j\}$ are increasing then $I_{j_1}\neq \emptyset$; 
if on the contrary the functions $\{f_j\}$ are decreasing then $I_{j_2}\neq \emptyset$. 
\item[-] Move the endpoint shared by $\overline{I}_{j_1}$ and $\overline{I}_{j_2}$, according to the monotonicity of the functions $\{f_j\}$: if the functions $\{f_j\}$ are increasing shrink $I_{j_1}$ and enlarge $I_{j_2}$; on the contrary if the functions $\{f_j\}$ are  decreasing shrink $I_{j_2}$ and enlarge $I_{j_1}$  (notice that,  according to the monotonicity of the $\{f_j\}$s, it is always possible to shrink these sets, thanks to the previous step). In both cases, by the continuity of the functions $\{f_j\}$, the value $f_{j_2}(I_{j_2})$ can be slightly increased such that $f_{j_2}(I_{j_2})< M$ while $f_{j_1}(I_{j_1})$ is decreased.
\item[-] If $\sharp(O)=1$ (that is, if $I_{j_1}$ was the only optimal interval) then also $M$ has decreased, which would be a contradiction with the optimality of $\{I_j\}$. Else if $\sharp(O)>1$, we remove $j_1$ from $O$ and the cardinality of $O$ is decreased of one unit. Therefore, repeat this procedure from the first step until $\sharp(O)=1$ to get the contradiction.
\end{itemize}

\smallskip
\emph{Sufficient condition.} Given a partition $\{I_j\}\in\mathcal C_n(I)$ satisfying the condition \eqref{optimality}, we want to prove that $\{I_j\}$ solves \eqref{problem}. Consider a minimizer $\{\widehat{I}_j\}\in\mathcal C_n(I)$ for problem \eqref{problem}. By Lemma~\ref{lemma:incl} there exist two indeces $j_1$ and $j_2$ such that $I_{j_1}\subseteq \widehat{I}_{j_1}$ and $\widehat{I}_{j_2}\subseteq I_{j_2}$.
Then, by \eqref{optimality}, according to the monotonicity of the functions $f_j$'s either
\begin{equation}\label{abstract1}
\max_{1\leq j\leq n} f_j(I_j)=f_{j_1}(I_{j_1})\leq f_{j_1}(\widehat{I}_{j_1})\leq \max_{1\leq j\leq n} f_j(\widehat I_j)
\end{equation}
or
\begin{equation}\label{abstract2}
\max_{1\leq j\leq n} f_j(I_j)= f_{j_2}(I_{j_2})\leq f_{j_2}(\widehat{I}_{j_2})\leq \max_{1\leq j\leq n} f_j(\widehat I_j).
\end{equation}
In both cases the minimality of $\{\widehat I_j\}$ implies that also $\{I_j\}$ solves \eqref{problem}.

\smallskip
\emph{Uniqueness.}
Let $\{I_j\}$ and $\{\widehat{I}_j\}$ be two distinct partitions in $\mathcal C_n(I)$ that minimize \eqref{problem}. From the necessary condition we know that these minimizers satisfy \eqref{optimality}.
If $\{I_j\}$ and $\{\widehat{I}_j\}$ are distinct partitions, by Lemma~\ref{lemma:incl} there exist two indeces $j_1$ and $j_2$ such that $I_{j_1}\subsetneq \widehat{I}_{j_1}$ and  $\widehat I_{j_2}\subsetneq {I}_{j_2}$. But arguing as for the sufficient condition, from \eqref{optimality} and the {strict} monotonicity of the functions $\{f_j\}$, either \eqref{abstract1} or \eqref{abstract2} holds with a strict inequality; this would violate the equality $\max_{1\leq j\leq n} f_j(I_j)=\max_{1\leq j\leq n} f_j(\widehat I_j)$. Then, the only possibility is that $\{I_j\}$ and $\{\widehat{I}_j\}$ are the same partition.
\end{proof}

For the existence of a solution lower semicontinuity of the functions would suffice; however, for the optimality condition the continuity is necessary. The procedure in the necessary condition can also be iterated to construct partitions that approximate the optimal one.

For particular families of set-function the solution of \eqref{problem} can be explicitly identified.

\begin{corollary}[Uniform partition]\label{cor.uniform}
Let $g\colon [0,+\infty]\to [0,+\infty]$ be  a strictly monotone continuous function. Let $n\in\mathbb N$ and $\{f_j\}_{j=1}^n$ be the family of set-functions such that $f_j=g\circ \mathcal L$, namely $f(J)=g(\mathcal L(J))$ for all intervals $J\subseteq I$.
Then the minimax problem \eqref{problem} is uniquely solved by the \emph{uniform partition}:
\[
I_j=\left(\frac{j-1}{n},\frac{j}{n}\right)\quad  \text{for all $j=1,\dots , n$.}
\]
\end{corollary}
\begin{proof}
The hypothesis on $g$ with Remark~\ref{rem.com} ensure the family $\{f_j\}$ to fulfill the assumptions of Theorem~\ref{teo.main}.
Then the optimality condition \eqref{optimality} reads as $g(\mathcal L(I_{j_1}))=g(\mathcal L(I_{j_2}))$ for all $1\leq j_1,j_2\leq n$. By exploiting the invertibility of the function $g$ one deduces $\mathcal L(I_{j_1})=\mathcal L(I_{j_2})$ for all $1\leq j_1,j_2\leq n$. This  concludes the proof.
\end{proof}

\begin{remark}\label{rem.inv}
If $\{f_j\}$ is a family of set-functions satisfying the hypothesis of Theorem~\ref{teo.main}, then also the family $\{1/f_j\}$ satisfy these hypothesis with reverse monotonicity. Then problem
$$\min_{\{I_j\}\in \mathcal C_n(I)}\max_{1\leq j\leq n} \frac{1}{f_j(I_j)}$$
is equivalent to \eqref{problem}, in the sense that these two problems have the same solutions. Indeed the optimality condition for this problem is
$$\frac{1}{f_{j_1}(I_{j_1})}=\frac{1}{f_{j_2}(I_{j_2})}, \quad \text{for all $1\leq j_1,j_2\leq n$},$$
which is clearly equivalent to \eqref{optimality}.
\end{remark}

Another interesting application is the equivalence of minimax and maximin partition problems.
\begin{corollary}[Maximin partition problems]\label{cor.maxmin}
Let $n\in\mathbb N$ and $\{f_j\}_{j=1}^n$ be a family of set-functions such that, for every $1\leq j\leq n$,
\begin{itemize}
\item $f_j$ is continuous and monotone;
\item $f_j$ and $f_{j+1}$ are compatible.
\end{itemize}
There exists a solution to the maximin partition problem \eqref{problem2}. A partition $\{I_j\}\in \mathcal C_n(I)$ solves \eqref{problem2} if and only if
\begin{equation*}
f_{j_1}(I_{j_1})=f_{j_2}(I_{j_2})\quad  \text{for all $1\leq j_1,j_2\leq n$.}
\end{equation*}
If, in addition, all the set-functions $f_j$ are \emph{strictly} monotone, then the solution is unique.
\end{corollary}
\begin{proof}
From the equality 
\[
\frac{1}{\displaystyle \max_{\{I_j\}\in \mathcal C_n(I)}\min_{1\leq j\leq n} f_j(I_j)}=\min_{\{I_j\}\in \mathcal C_n(I)}\max_{1\leq j\leq n} \frac{1}{f_j(I_j)},
\]
one can see that the two problems are equivalent, admitting the same solutions. By Remark~\ref{rem.inv} the thesis follows.
\end{proof}

In view of the previous results several optimal partition problems can be considered.

\begin{example}[Optimal partitioning for measures]
Let $n\in\mathbb N$ and for every $1\leq j\leq n$ let $\mu_j$ be a non-atomic probability measure over $\overline I$. By Remark~\ref{rem.meas} one can let in \eqref{problem} $f_j=\mu_j$ and study
\[
\min_{\{I_j\}\in \mathcal C_n(I)}\max_{1\leq j\leq n}\mu_j(I_j).
\]
and also the dual problem
\[
\max_{\{I_j\}\in \mathcal C_n(I)}\min_{1\leq j\leq n}\mu_j(I_j).
\]
By Theorem~\ref{teo.main} and Corollary~\ref{cor.maxmin} every optimal partition must satisfy $\mu_{j_1}(I_{j_1})=\mu_{j_2}(I_{j_2})$ for all $1\leq j_1,j_2\leq n$. The solution is unique whenever the measures $\mu_j$ are strictly increasing.
These partitioning are classical and linked to problems of \emph{fair division} of an object, see for instance \cite{dalewi, elhike, lewi}.
A famous cake-cutting interpretation by Dubins and Spanier \cite{dubspa} is as follows. Suppose a cake $I$ is to be divided among $n$ people whose values $\mu_j$ of different portions of the cake may differ (i.e., $\mu_j(I_j)$ represents the value of the piece $I_j$ to person $j$). In \cite{dubspa} it was shown that if not all the values $\mu_j$ are identical, there is always a partition of $I$ into $n$ pieces so that each person receives a piece he values strictly more than $1/n$, what is known as \emph{equitable fair-cutting}.
\end{example}

\begin{example}[Optimal location problems]
By Remark~\ref{rem.dist} one can consider the optimal partition problems
\[
\min_{\{I_j\}\in \mathcal C_n(I)}\max_{1\leq j\leq n}\int_{I_j} \rho(x)\mathrm{d}_{\partial I_j}(x)^r dx
\]
or 
\[
\min_{\{I_j\}\in \mathcal C_n(I)}\max_{1\leq j\leq n} \max_{x\in I_j}(\rho(x)\mathrm{d}_{\partial I_j}(x)).
\]
By Theorem~\ref{teo.main} the solutions of these problems are unique: the former satisfies the optimality conditions $\int_{I_{j_1}} \rho(x)\mathrm{d}_{\partial I_{j_1}}(x)^r dx=\int_{I_{j_2}} \rho(x)\mathrm{d}_{\partial I_{j_2}}(x)^r dx$ while the latter $\max_{x\in I_{j_1}}(\rho(x)\mathrm{d}_{\partial I_{j_1}}(x))=\max_{x\in I_{j_2}}(\rho(x)\mathrm{d}_{\partial I_{j_2}}(x))$ for all $1\leq j_1,j_2\leq n$.
The latter problem can also be rewritten as
\begin{equation}\label{facility}
\min_{\{I_j\}\in \mathcal C_n(I)} \max_{x\in I}(\rho(x)\mathrm{d}_{E}(x))
\end{equation}
where $E:=\cup_j\partial I_j$ is a set made up of $n-1$ points in $\overline I$. Notice that these optimal partition problems are tightly related to some well-known \emph{optimal location problems} (see \cite{bojima,busava, suzdre,suzoka} and references therein) where the points (and not the partitions) are the unknowns in the optimization problem. In this terms problem \eqref{facility} may have the following  interpretation from \emph{urban planning}: if $\rho$ denote the density of a population that live along a road $I$ and if the $n-1$ points $\{x_j\}\in E$ represents some facilities (e.g., petrol stations, metro stations, supermarkets, \dots) that can be builded along the road $I$, with problem \eqref{facility} the builder is looking for the best position to place the facilities so as to minimize the (average or maximum) distance the people have to cover to reach the nearest facility. 
\end{example}

\begin{example}[Spectral partitions for Sturm-Liouville problems]
Remarkable optimal partition problems are the ones concerning the eigenvalues of Sturm-Liouville problems. By Remark~\ref{rem.eigen} we can let $f_j=\lambda_{n_j}$ for some $n_j\in\mathbb N$ and study
\[
\min_{\{I_j\}\in \mathcal C_n(I)} \max_{1\leq j\leq n}\lambda_{n_j}(I_j),
\]
or also, by Corollary~\ref{cor.maxmin}, the dual problem
\[
\max_{\{I_j\}\in \mathcal C_n(I)} \min_{1\leq j\leq n}{\lambda_{n_j}(I_j)}.
\]
Theorem~\ref{teo.main} then applies: for both problems the solution is unique and  satisfies $\lambda_{n_{j_1}}(I_{j_1})=\lambda_{n_{j_2}}(I_{j_2})$ for all $1\leq j_1,j_2\leq n$. Versions of these problems in higher dimension for the Laplace operator have applications to the phenomenon of the spatial segregation in reaction-diffusion systems \cite{bfvv,caflin,coteve, hehote}. 
 
By the way, many others \emph{spectral partition problems} can be considered, such as the ones involving the eigenvalues of \emph{singular operators}, \emph{nonlinear operators} (i.e., the $p$-Laplacian for $1<p<+\infty$), \emph{higher order} operators (i.e., the bi-Laplacian) or \emph{fractional operators} (i.e., the $s$-Laplacian for $0<s<1$). Notice that the \emph{torsional rigidity} (also called the \emph{compliance}) may be considered as well.
\end{example}

At last, we point out that one can also consider different kind of set-functionals on each interval of the partition (e.g. for $n=1$ a measure on the first interval and an eigenvalue on the second interval) leading to \emph{mixed} optimal partition problems.

\section{Asymptotic distribution of optimal partitions}\label{sec.3}
We discuss in this section the asymptotics of the optimal partitions of \eqref{problem} when the number $n$ of intervals of the partition goes to $+\infty$. For the sake of clarity we only consider the case $f_j=f$ for all $1\leq j\leq n$, for  a \emph{given} set-function $f$ independent of $n$; the more general situation of a family of functions $\{f_j^n\}$ that may change from time to time as $n$ increases should be treated with some effort, but with no substantially new ideas. We thus focus on the limiting behaviour as $n\to\infty$ of the minimum problem
\begin{equation}\label{problembis}
\min_{\{I_j\}\in \mathcal C_n(I)}\max_{1\leq j\leq n} f(I_j).
\end{equation}
Since now the parameter $n$ varies it is convenient to emphasize the dependence of any optimal partitions on $n$ by a superscript, namely for every $n\in\mathbb N$ we denote by $\{I_j^n\}$ a solution of \eqref{problembis}. 
Then, to keep track of the asymptotic distribution of the optimal partitions, for every $n\in\mathbb N$ we associate to a minimizer $\{I_j^n\}\in\mathcal C_n(I)$ of \eqref{problembis} the probability measure $\mu_{n}\in\mathcal P(\overline I)$ defined as
\begin{equation}\label{meas}
 \mu_{n}=\frac{1}{n-1}\sum_{1\leq j\leq n-1}\delta_{x_j^n},
\end{equation}
where $\{x_j^n\}$ is the family of points that identify the optimal partition $\{I_j^n\}$ and $\delta_x$ denotes the Dirac delta supported at a point $x\in I$ (the normalization factor $n-1$ guarantees $\mu_n$ to be a probability measure). Notice that this association is a common strategy that has been used in similar contexts (see for instance \cite{bojima, busava, tilzuc,tilzuc3, tilzuc2}).
Now if $n\to\infty$, by the compactness of the space of probability measures $\mathcal P(\overline I)$, we may assume that, up to subsequences, the probability measures $\{\mu_n\}$ defined in \eqref{meas} weakly* converge to some probability measure $\mu$. The purpose of this section is therefore the identification of such a limiting measure $\mu$.

For the asymptotics we need the following notation. 
\begin{definition}[Radon-Nikodym condition]\label{def.rn}
A monotone set-function $f$ has the Radon-Nikodym property if there exists a function $s\in L^1(I)$, $s>0$ a.e. on $I$, such that if $f$ is increasing there holds 
\begin{equation}\label{eq.rn}
\lim_{J\downarrow x}\frac{f(J)}{\mathcal L( J)}= s(x), \quad \text{for a.e. $x\in I$};
\end{equation}
if $f$ is decreasing there holds
\[
\lim_{J\downarrow x}\frac{1}{f(J)\mathcal L( J)}= s(x), \quad \text{for a.e. $x\in I$}.
\]
\end{definition}

The strict positivity of $s$ over the entire $I$ could be relaxed to hold only on a subset of positive Lebesgue measure.

\begin{definition}[Domination]\label{def.dom} A monotone set-function $f$ is said \emph{dominated} if there exists a function $d\in L^1(I)$ such that if $f$ is increasing there holds 
\[
\sum_{1\leq j\leq n} \frac{f(I_j)}{\mathcal L(I_j)}\chi_{I_j}(x)\leq d(x), \quad \text{for all $n\in \mathbb N$, all $\{I_j\}\in\mathcal C_n$, and for a.e. $x\in I$};
\]
if $f$ is decreasing there holds 
\begin{equation*}
\sum_{1\leq j\leq n} \frac{1}{f(I_j)\mathcal L(I_j)}\chi_{I_j}(x)\leq d(x), \quad \text{for all $n\in \mathbb N$, all $\{I_j\}\in\mathcal C_n$, and for a.e. $x\in I$}.
\end{equation*}
\end{definition}

By the positivity of $f$ the function $s$ and $d$ are positive as well.

\begin{remark}\label{rem.asy}
We assume the previous two ratios in Definition~\ref{def.dom} to be zero when they are not well-defined; this problem only occurs when $I_j$ is unbounded, $f(I_j)=+\infty$ as $f$ increasing, and $f(I_j)=0$ as $f$ decreasing (notice that in Definition~\ref{def.rn} this problem does not subsist since $J$ can always be taken bounded).
\end{remark}

\begin{remark}[Measures]
The Radon-Nikodym condition in Definition~\ref{def.rn} is inspired by the classical one that holds for measures. It is well known \cite[Theorems 7.10]{rudin}) that if $f=\mu$ for some positive finite measure $\mu$ on $I$ then \eqref{eq.rn} holds for a.e. point $x\in I$ with $s=\mu^a$, the \emph{absolutely continuous} part of $\mu$ with respect to the Lebesgue measure. The points for which this limit exists are the so-called \emph{Lebesgue points} of $\mu$. Notice that if $\mu$ is arbitrary we may also have $\mu^a=0$ on $I$.
\end{remark}

\begin{remark}[Eigenvalues]\label{rem.eigen2}
Remarkable set-functions satisfying  Definitions \ref{def.rn} and \ref{def.dom} are the eigenvalues of Sturm-Liouville operators with a \emph{variational} structure. Let $r,t\in (1,+\infty)$, $\beta>1$ and $p,q,w\in L^\infty(I)$ such that $1/\beta\leq p,w\leq \beta$ and $0\leq q\leq \beta$. For every interval $J\subseteq I$, in case $I$ is bounded one can define the eigenvalue
\[
\lambda_1^{r,t}(J):= \min_{\begin{subarray}{c}u\in W^{1,r}_0(D)\\ u\neq 0\end{subarray}}\bigg\{\frac{\int_{J} p(x)u'(x)^r dx+ \int_J q(x) u(x)^r dx}{(\int_{J}w(x)u(x)^t dx)^{{r}/{t}}}\bigg\}.
\]
For $r=t=2$ coincides with the first eigenvalue of the Sturm-Liouville problem considered in Remark~\ref{rem.eigen}, while for $r=2$ and $t=1$ it is the so-called \emph{compliance} functional (the other values correspond to some eigenvalues of non-linear operators).
Then let $f$ be the set function $(\lambda_1^{r,t})^{1/r}$.
For Definition~\ref{def.rn} we can refer to the limit proved in \cite[Lemma~2.1]{tilzuc3} where it has been identified the function $s=\sqrt{w/p}$. Definition~\ref{def.dom} holds thanks to the boundedness assumptions on $I$ and on the coefficients $p,q,w$. Indeed, by \eqref{eigenvalue} we have
\[
\sum_{j=1}^n \frac{1}{\lambda_1(I_j)^{1/r}\mathcal L (I_j)}\chi_{I_j}(x)\leq C{\beta^{1/r+1/t}} \quad \text{for a.e. $x\in I$,}
\]
for a suitable constant $C$, depending only on $I$, $r$, and $t$.
\end{remark}

The main result we prove in this section is the following one.   

\begin{theorem}[Asymptotic distribution]\label{teo.main2}
Let $f$ be a monotone continuous set-function with the Radon-Nikodym property and which is dominated. For every $n\in \mathbb N$ let $\{I_j^n\}$ denote an optimal partition of \eqref{problembis}. As $n \to \infty$, the probability measures $\{\mu_{n}\}_n$ associated to $\{I_j^n\}$ via \eqref{meas} converge in the weak* topology on $\mathcal P(\overline I)$ to the probability measure $\mu$, absolutely continuous with respect to the Lebesgue measure, with density given by
\begin{equation}\label{minimum}
\mu^a(x)=\frac{s(x)}{\int_I s(t)dt}.
\end{equation}
In particular, for every open set $J\subseteq I$,
\begin{equation}\label{portion}
\lim_{n\to\infty} \frac{\sharp (\{I_j^n\}\cap J)}{n}=\frac{\int_J s(x)dx}{\int_I s(x)dx}.
\end{equation}
Moreover, when $f$ is increasing then
\begin{equation}\label{value}
 \lim_{n\to\infty} n\max_{1\leq j\leq n} f(I_j^n)=\int_I s(x)dx;
\end{equation}
when $f$ is decreasing then
\begin{equation}\label{value2}
 \lim_{n\to\infty} \frac{1}{n}\max_{1\leq j\leq n}f(I_j^n)=\frac{1}{\int_I s(x)dx}.
\end{equation}
\end{theorem}

\begin{remark}\label{rem.inc}
Since any minimax problem of decreasing functions can always be reconducted to an \emph{equivalent} minimax problem of increasing functions (just by passing to reciprocals as in Remark~\ref{rem.inv}) it suffices to prove the asymptotics only for increasing functions.
\end{remark}

The first preliminary result we need is the following lemma that guarantees the existence of the limit as $n\to\infty$ of the minimum value of \eqref{problembis}.
\begin{lemma}\label{lem.value}
Let $f$ be an increasing continuous set-function with the Radon-Nikodym property and which is dominated. Fix an interval $J\subseteq I$. For every $n\in \mathbb N$, let $\{I_j^n\}_j$ denote the optimal partition of \eqref{problembis}  and $k_n\geq 1$ denote the number of intervals of the partition $\{I_j^n\}$ with non-empty intersection with $J$.
In case $f$ is increasing then 
\begin{equation}\label{bb}
\liminf_{n\to\infty} k_n \max_{1\leq i\leq k_n} f(I_{j_i}^n)\geq \int_J s(x)dx.
\end{equation}
Moreover, when $J=I$ then $k_n=n$ and
\begin{equation}\label{bb2}
\lim_{n\to\infty} n \max_{1\leq j\leq n} f(I_j^n)= \int_I s(x)dx.
\end{equation}
\end{lemma}
\begin{proof}
From the Radon-Nikodym property we have
\[
\lim_{n\to\infty} \sum_{1\leq j\leq n} \frac{f(I_j^n)}{\mathcal L(I_j^n)}\chi_{I_j^n}(x)= s(x),\quad \text{for all $n\in \mathbb N$, all $\{I_j\}\in\mathcal C_n$, and for a.e. $x\in I$}.
\]
Since these functions are dominated by assumption, we can apply the Lebesgue dominated convergence theorem to obtain
\[
\lim_{n\to\infty} \sum_{1\leq j\leq n} \frac{f(I_j^n)}{\mathcal L(I_j^n)}{\mathcal L(I_j^n\cap J)}=\int_Js(x)dx,
\]
where recalling Remark~\ref{rem.asy} the ratio ${f(I_j^n)}/{\mathcal L(I_j^n)}$ has to be meant zero when not well defined.
Then \eqref{bb} immediately follows by $0<\mathcal L(I_j^n\cap J)\leq \mathcal L(I_j^n)$ for $k_n$ intervals $I_j$ (when $I$ is unbounded $1/\mathcal L (I_j^n)=0$ for at most two intervals of the partitions and they are negligible in the asymptotics in $n$ since by the Radon-Nikodym property $k_n\to\infty$ as $n\to\infty$). 
Similarly, when the interval $J$ is replaced by $I$, we have $k_n=n$ and $\mathcal L(I_j^n\cap I)= \mathcal L(I_j^n)$, which combined with the optimality condition \eqref{optimality} yield the limit \eqref{bb2} (again when $I$ is unbounded $1/\mathcal L (I_j^n)=0$ for at most two intervals of the partitions and they are negligible in the asymptotics in $n$).
\end{proof}

We then need a lower bound for the minimum value of \eqref{problem} as $n\to\infty$. This provides information on the absolutely continuous part of $\mu$.

\begin{lemma}\label{lem.value2}
Let $f$ be an increasing continuous set-function with the Radon-Nikodym property and which is dominated. For every $n\in \mathbb N$, let $\{I_j^n\}_j$ denote the optimal partition of \eqref{problembis}. Assume that the measures $\{\mu_n\}$ defined by \eqref{meas} weak* converge to some measure $\mu\in\mathcal P(\overline I)$ and let $\mu^a$ denote the absolutely continuous part of $\mu$ with respect to the Lebesgue measure. Then 
\begin{equation}\label{aa}
\int_I s(x)dx\geq \esssup_{x\in I} \frac{s(x)}{\mu^a(x)}.
\end{equation}
Moreover, $\mu^a(x)>0$ a.e. $x\in I$.
\end{lemma}
\begin{proof}
Fix $\epsilon>0$, a point $x_0\in I$ be a Lebesgue point for $s$ and $\mu^a$, and an interval $J\subseteq I$ centered at $x_0$.
Let $k_n\geq 1$ denote the number of intervals of the partition $\{I_j^n\}$ with non-empty intersection with $J$ so that by \eqref{meas}
\begin{equation*}
\mu_n(J)= \frac{k_n-1}{n-1},
\end{equation*}
and then
\[
n \max_{1\leq j\leq n} f(I_j^n)\geq \frac{n}{n-1}\frac{(k_n-1)\max_{1\leq j\leq k_n} f(I_j^n)}{\mu_n(J)+\epsilon}.
\]
The role of $\epsilon$ is to avoid possible vanishing denominators. By letting $n\to\infty$ in this inequality and combining $\limsup_{n\to\infty} \mu_n(J)\leq \mu(\overline J)$ with \eqref{bb} and \eqref{bb2} one obtains that
\[
\int_I s(x)dx\geq \frac{\int_J s(x)}{\mu(\overline J)+\epsilon}.
\]
Letting now $J$ shrink towards $x_0$, from Radon-Nikodym theorem, we find that
\[
\int_I s(x)dx\geq \frac{s(x_0)}{\mu^a(x_0)+\epsilon}.
\]
Since a.e. point $x_0\in I$ is a Lebesgue point for $s$ and $\mu^a$, from the arbitrariness of $x_0$ and of $\epsilon$ we obtain \eqref{aa}. The last assertion follows from \eqref{aa} and the integrability of the function $s$.
\end{proof}

We are now ready for proving Theorem~\ref{teo.main2}.
\begin{proof}[Proof of Theorem~\ref{teo.main2}]
As previously observed in Remark~\ref{rem.inc} if suffices to prove the theorem when $f$ is increasing. 
By Lemma~\ref{lem.value2} and the fact that $\mu$ is a probability measure we obtain
\[
 \esssup_{x\in I} \frac{s(x)}{\mu^a(x)}\leq \int_I \frac{s(x)}{\mu^a(x)}{\mu^a(x)}dx\leq \esssup_{x\in I} \frac{s(x)}{\mu^a(x)}\int_I\mu^a(x)dx\leq \esssup_{x\in I} \frac{s(x)}{\mu^a(x)}.
\]
and therefore, the chain of inequalities is in fact a chain of equalities. By studying the equality cases one deduces that 
\begin{equation*}
\mu^a(x)=\frac{s(x)}{\int_I s(t)dt},
\end{equation*}
which implies that the limiting measure $d\mu=\mu^a(x)dx $. By definition of weak* convergence of measures and \eqref{meas} we have
\[
\lim_{n\to\infty} \frac{\sharp (\{I_j^n\}\cap J)-1}{n-1}=\frac{\int_J s(x)dx}{\int_I s(x)dx},
\]
which is clearly equivalent to the limit in \eqref{portion}.
The limit \eqref{value} is a consequence of Lemma~\ref{lem.value}.
\end{proof}

\begin{remark}[$\Gamma$-convergence]
Let $F_n\colon\mathcal P(\overline I)\to [0,+\infty]$ and $F_{\infty}\colon\mathcal P(\overline I)\to [0,+\infty]$ be the functionals defined as
\begin{equation*}
F_n(\mu):=
\begin{cases}
\displaystyle n \max_{1\leq j\leq n}f(I_j)\quad &\text{if $\mu=\mu_n$ as in \eqref{meas} with $\{I_j^n\}\in\mathcal C_n(I)$ solution of \eqref{problem},}\\
\quad\; +\infty \quad &\text{otherwise},
\end{cases}
\end{equation*}
and
\begin{equation*}
F_\infty(\mu):=
\begin{cases}
\displaystyle \int_I s(x) dx&\text{if $\mu=s(x)/\int_I s(x)dx$,}\\
\quad \quad +\infty  &\text{otherwise}.
\end{cases}
\end{equation*} 
As a byproduct of Lemmas \ref{lem.value} and \ref{lem.value2} we deduce that as $n\to\infty$ the functionals $F_n$ $\Gamma$-converge with respect to the weak* convergence in the space of probability measures $\mathcal P(\overline I)$ to the functional $F_\infty$ (see \cite{dalmaso} for more information on $\Gamma$-convergence theory). 
This $\Gamma$-limit is tightly related to some supremal functionals derived in \cite{tilzuc} and \cite{tilzuc2} for a maximization problem of the first eigenvalue of an elliptic operator in two dimensions. 
\end{remark}

\section{A maximization problem for Sturm-Liouville eigenvalues}\label{sec.5}

In this section we assume $I$ to be \emph{bounded} and $p,q,w\in L^\infty(I)$ be functions such that $1/\beta\leq p,w\leq \beta$ and $0\leq q\leq \beta$ for some constant $\beta\geq 1$. As previously notices in Remark~\ref{rem.eigen2}, under these assumptions to every open set $J\subseteq I$, different from the empty set,  we can characterize the first eigenvalue of the Sturm-Liouville operator with coefficients $p,q,w$ by means of its \emph{variational} structure:
\begin{equation}\label{eigenvalue}
\lambda_1(J)= \min_{\begin{subarray}{c}u\in H^{1}_0(J)\\ u\neq 0\end{subarray}}\bigg\{\frac{\int_{J} p(x)u'(x)^2 dx+ \int_J q(x) u(x)^2 dx}{\int_{J}w(x)u(x)^2 dx}\bigg\},
\end{equation}
and the \emph{first eigenfuction} $u_1$ as solution to
\begin{equation*}
-(p u_{1}')'+qu_{1}=\la_1({J})w u_{1},
\end{equation*}
with the Dirichlet boundary  condition $u_{1}=0$ on $\partial J$ (notice that $u_1$ is unique, when $J$ is connected, up to a multiplicative factor). By convention, when $J=\emptyset$ we set $\la_1(J)=+\infty$ (this is not merely a convention but it is also consistent with the continuity required by Definition~\ref{def.con}).
In this way we may consider $f_j=(\lambda_1)^{{1}/{2}}$ so that the
maximin problem \eqref{problem2} writes as
\begin{equation}\label{problem3}
\max_{\{I_j\}\in \mathcal C_n(I)}\min_{1\leq j\leq n} \lambda_1(I_j)^{1/2}.
\end{equation}
If $\Sigma_{n}=\{x_1^n,x_2^n,\dots, x_{n-1}^n\}$ denote the set of $(n-1)$ points in $I$ that identify the partition $\{I_j^n\}_{j=1}^n$, then problem \eqref{problem2} admits a peculiar formulation: the functional to be maximized is related to the first eigenvalue $\la_1(I\setminus \Sigma_{n})$ of the disconnected set $I\setminus \Sigma_{n}$, as defined in \eqref{eigenvalue} with $J=I\setminus \Sigma_{n}$.
Indeed, it is well known that the first eigenvalue $\la(I\setminus \Sigma_{n})$ splits over connected components
\begin{equation*}
\la(I\setminus \Sigma_{n})=\min_{1\leq j \leq n+1}\la(I_j^n),
\end{equation*}
therefore \eqref{problem3} admits the equivalent \emph{maximizing} formulation: 
\begin{equation}\label{probmax}
 \max_{\Sigma_n} \la(I\setminus \Sigma_{n})^{1/2},
\end{equation}
where the maximum is computed among all sets $\Sigma_{n}$ made up of $(n-1)$ (non necessarily distinct) points.
An interesting discussion about this problem was raised by Courant and Hilbert in \cite[pp.~463--464]{couhil}.
Two dimensional versions of this maximization problem \eqref{probmax} have also been investigate in some recent works, see \cite{henzuc,tilzuc,tilzuc2}. 

A possible physical interpretations of problem \eqref{probmax} is as follows.
In acoustics, $I$ represents a non-homogeneous string (of density $w$, Young modulus $p$ and subjected to the potential $q$) fixed at its endpoints with a frequency of vibration proportional to the first eigenvalue ${\la_1(I)^{1/2}}$. 
By adding $n-1$ extra points (nails) in the middle, where the string will be supplementarily fixed, the whole string will then vibrate according to the $n$ independent substrings $I_j^n$. 
Therefore, we are asking where best to place the points so as to reinforce a string and maximize its fundamental frequency of vibration $\lambda(I\setminus \Sigma_n)^{1/2}$.

In view of Remarks~\ref{rem.eigen} and \ref{rem.eigen2}, we can rephrase Corollary~\ref{cor.maxmin} and Theorem~\ref{teo.main2} to obtain information on the maximization problem \eqref{probmax}.

\begin{theorem}\label{teo.2}
For every $n\in\mathbb{N}$, there exists a \emph{unique} maximizer $\Sigma_{n}$ of \eqref{probmax}. Moreover, a configuration $\Sigma_{n}$ is the maximizer of \eqref{probmax} if and only 
\begin{equation}\label{eq.optimal}
\la_1(I\setminus \Sigma_{n})=\la_1(I_j^n) \quad \text{for every $1\leq j\leq n$}.
\end{equation}
 
As $n \to \infty$, if $\Sigma_{n}$ is the maximizing sequence of problem \eqref{probmax}, then the probability measures $\mu_{n}$ defined by \eqref{meas} converge
in the weak* topology on $\mathcal P(\overline I)$ to the probability measure $\mu$, absolutely continuous with respect to the Lebesgue measure, with density 
given in \eqref{minimum}. In particular \eqref{portion} holds and
\[
 \lim_{n\to\infty} \frac{n}{\la_1 (I\setminus \Sigma_{n})^{1/2}}=\frac{1}{\pi}
 \int_I\sqrt{\frac{w(x)}{p(x)}}dx.
\]
\end{theorem}

This theorem say that in order to reinforce a non-homogeneous string, there exists a unique maximal configuration which is governed by an optimality condition depending on the three coefficients $p, q, w$ of the Sturm-Liouville operator. 
Moreover, as $n\to\infty$, it is convenient to concentrate the points in those regions of the string at higher density $w$ and lower Young modulus $p$, with a density $\mu^a$  given by \eqref{minimum}, namely proportional to $\sqrt{w/p}$. Observe that the limiting measure $\mu^a(x) dx$ is absolutely continuous with respect to the Lebesgue measure, and that it does not depend on the potential $q$ (thought the optimality condition \eqref{eq.optimal} does).

If $p$ and $w$ are proportionally equivalents, that is $p(x)=c\cdot w(x)$ for a constant $c>0$, we obtain as a limit for the sequence of minimizers the uniform measure on $I$. This happens for example in the case of constant coefficients, for which problem \eqref{probmax} can be explicitly solved. Indeed, if $p, q, w$  are constants and if $J\subseteq I$ is an interval, the eigenvalue defined in \eqref{eigenvalue} has the explicit representation
\begin{equation*}
\la_1(J)=\frac{p}{w}\frac{\pi^2}{\mathcal L(J)^2}+q,
\end{equation*}
which thanks to Corollary~\ref{cor.uniform} implies that the set of equispaced points $\{j/n:\; 1\leq j\leq n-1\}$ is the unique solution to \eqref{probmax}.

\smallskip
The optimality condition \eqref{eq.optimal} stated in Theorem~\ref{teo.2} has a deeper interpretation in terms of higher eigenvalues.
If $n\in \mathbb N$, similarly to \eqref{eigenvalue} to any open set $J\subseteq I$, different from the empty set, we may associate the \emph{$n$-th eigenvalue} of $J$ of the Sturm-Liouville operator with coefficients $p$, $q$, $w$, via the \emph{Courant-Fischer-Weyl min-max principle} \cite{couhil}
\begin{equation}\label{eigenvaluen}
\la_n(J):= \min_{\begin{subarray}{c}U_n\subseteq H^1_0(J)\\ \text{$U_n$ subspace of dim $n$}\end{subarray}} \max_{\begin{subarray}{c}u\in U_n\\ u\neq 0\end{subarray}} \frac{\int_{J} p(x)u'(x)^2dx+ \int_J q(x) u(x)^2 dx}{\int_{J}w(x)u(x)^2 dx},
\end{equation}
and the \emph{$n$-th eigenfuction} $u_n$ a solution to
\begin{equation}\label{ELn}
-(p u_{n}')'+qu_{n}=\la_n({J})w u_{n},
\end{equation}
with the Dirichlet boundary  condition $u_{n}=0$ on $\partial J$. Again by convention, when $J=\emptyset$ we set $\la_n(J)=+\infty$. Notice that the minimum in \eqref{eigenvaluen} is achieved by choosing as $U_n$ the space spanned by the $n$-th first eigenfunctions and that, by orthogonality, $u_1$ is the unique eigenfunction of constant sign when $J$ is connected. A classical result of Sturm states that $u_{n}$ has exactly $(n-1)$ distinct zeros inside $I$ (see \cite{stu, stu2} and also the recent paper \cite{berhel}). The set $Z_n$ of these zeros induce in $I$ a partition $\{I_j^n\}$ made up of $n$ subintervals in which $u_{n}$ has constant sign. Moreover, since $u_{n}$ solves \eqref{eigenvaluen} in $I_j^n$ and vanishes on $\partial I_j^n$, it must be the first eigenfunction in $I_j^n$ with $\la_n(I)=\la_1(I_j^n)$ for every $j=1, \dots, n$. These considerations suggest the following interpretation of the maximizer to \eqref{probmax}.

\begin{theorem}\label{teo.3}
Let $n\in\mathbb N$. A configuration $\Sigma_{n}$ solves \eqref{probmax} if and only if it is the set of the zeros of the $n$-th eigenfunction of the Sturm-Liouville equation \eqref{ELn}.  
\end{theorem}
\begin{proof}
By Sturm's theorem the set $Z_{n}$ of the zeros of the $n$-th eigenfunction of the Sturm-Liouville equation \eqref{ELn} satisfies $\la_n(I)=\lambda_1(I_j^n)$ for all $j=1,\dots, n$. This condition is verified if and only if \eqref{eq.optimal} holds. Therefore $Z_n=\Sigma_{n}$.
\end{proof}

By combining Theorems \ref{teo.2} with \ref{teo.3} we may also derive some classical results of Sturm-Liouville theory.

\begin{theorem}\label{teo.4}
The following results hold.
\begin{itemize}
\item[(i)]\emph{(Distribution of the zeros of eigenfunctions).} The set of zeros $Z_{n}$ of the $n$-eigenfunction $u_{n}$ in  \eqref{ELn} distributes inside an open interval $J\subseteq I$ as
\[
\lim_{n\to\infty} \frac{\sharp (Z_n\cap J)}{n}= \frac{\int_J \sqrt{w(x)/p(x)}dx}{\int_I \sqrt{w(x)/p(x)}dx}.
\]
\item[(ii)] \emph{(Asymptotics of the eigenvalues).} The eigenvalues \eqref{eigenvaluen} satisfy
\[
\lim_{n\to\infty} \frac{n}{\la_{n}(I)^{1/2}}=\frac{1}{\pi}\int_I\sqrt{\frac{w(x)}{p(x)}}dx.
\]
\item[(iii)] \emph{(Weyl law for the eigenvalues).} The counting function $N$ that to every $\xi\in \mathbb R^+$ associates $N(\xi):=\sup\{n : \text{$\lambda_n$ is as in \eqref{eigenvaluen}}, \lambda_n \leq \xi\}$, namely the function that counts the eigenvalues \eqref{eigenvaluen} (counting their multiplicities) less than or equal to $\xi$,  satisfies 
\[
\lim_{\xi\to\infty} \frac{N(\xi)}{\xi^{1/2}}=\frac{1}{\pi}
 \int_I\sqrt{\frac{w(x)}{p(x)}}dx.
\]
\end{itemize}
\end{theorem}
\begin{proof}
Items (i) and (ii) are an immediate consequence of Theorem \ref{teo.3} with the asymptotics in Theorem \ref{teo.2}, see also \eqref{portion} and \eqref{value2}. To prove (iii) given $\xi\in\mathbb R^+$ let $n$ such that $n=N(\xi)$ such that $$\frac{n}{\la_{n+1}(I)^{1/2}}\leq \frac{N(\xi)}{\xi^{1/2}}\leq \frac{n}{\la_{n}(I)^{1/2}}.$$
Letting $\xi \to \infty$, since also $n\to \infty$, by the limit (ii) one obtains (iii).
\end{proof}

For item (i) of Theorem~\ref{teo.4} one can consult \cite[Chapter X]{ince}, and see also \cite[p.~314]{halmcl} for very good estimates on all the nodes.
Items (ii) and (iii) are also classical and can be found, for instance, in \cite[p.~415]{couhil} (the case of singular operators is contained in \cite{atkmin}, see also \cite[Equation (4.3.3)]{zettl}). We stress that our approach to recover classical results of Sturm-Liouville theory is new and related to purely variational results (such as the $\Gamma$-convergence theory, even if this is not directly used). Therefore, it could be used to exploit other asymptotic formulas; we aim at analyzing  this kind of topics in future works.

\subsection*{Acknowledgements}
This work is part of the Research Project INdAM for Young Researchers (Starting Grant) \emph{Optimal Shapes in Boundary Value Problems} and of the INdAM - GNAMPA Project 2018 \emph{Ottimizzazione Geometrica e Spettrale}.

\end{document}